\documentclass[]{amsart}
\usepackage{amsmath,amsthm,amssymb,enumerate,quiver,graphicx,hyperref}
\hypersetup{
	colorlinks,
	citecolor=black,
	filecolor=black,
	linkcolor=black,
	urlcolor=black,
}
\usepackage[
backend=bibtex,
style=numeric,
sorting=nty,
maxbibnames=99,
doi=false,
isbn=false
]{biblatex}
\theoremstyle{plain}
\newtheorem{thm}{Theorem}[section]
\newtheorem{lem}[thm]{Lemma}

\newtheorem{cor}[thm]{Corollary}
\newtheorem{prop}[thm]{Proposition}
\newtheorem{conj}[thm]{Conjecture}

\theoremstyle{definition}
\newtheorem{defin}[thm]{Definition}
\newtheorem{exam}[thm]{Example}
\newtheorem{rem}[thm]{Remark}

\newtheorem*{exam*}{Example}

\newcommand{\calV}{\mathcal{V}}
\newcommand{\calD}{\mathcal{D}}
\newcommand{\calP}{\mathcal{P}}
\newcommand{\calY}{\mathcal{Y}}

\newcommand{\T}{\mathbb{T}}
\newcommand{\m}{\mathfrak{m}}

\newcommand{\Hom}{\mathrm{Hom}}
\newcommand{\p}{\mathfrak{p}}
\newcommand{\subP}{_{(\p)}}
\newcommand{\Dsg}{\calD_\text{sg}}
\newcommand{\Spec}{\mathrm{Spec}}

\title[Nucleus and Support of Singularity Categories]{The Nucleus of a Compact Lie Group, and Support of Singularity Categories}
\author{Thomas Peirce}
\thanks{The author is supported by the Warwick Mathematics Institute Centre for Doctoral Training, and is grateful for funding by the UK EPSRC (Grant number: \texttt{EP/W523793/1}).}
\email{Thomas.Peirce@warwick.ac.uk}

\begin{document}
\begin{abstract}
	In this paper we adapt the notion of the nucleus defined by Benson, Carlson, and Robinson to compact Lie groups in non-modular characteristic. We show that it describes the singularities of the projective scheme of the cohomology of its classifying space. A notion of support for singularity categories of ring spectra (in the sense of Greenlees and Stevenson) is established, and is shown to be precisely the nucleus in this case, consistent with a conjecture of Benson and Greenlees for finite groups.
\end{abstract}
\maketitle

\section{Introduction}
In modular representation theory, the nucleus $\calY_G$ of a finite group $G$ over a field $k$ of positive characteristic, defined by Benson, Carlson, and Robinson in \cite{BCR}, is a homogeneous subvariety of the group variety that has equivalent \cite{Benson1995} characterisations:\begin{itemize}
	\item The ``structural'' definition as the union of subvarieties of the group variety associated to elementary Abelian $p$-subgroups whose centralisers are not $p$-nilpotent.
	\item The ``representation theoretic'' definition as the union of subvarieties associated to non-projective modules $M$ in the principal block with $H^*_\text{gp}(G,M)$ vanishing in positive codegrees.
\end{itemize}
The vanishing of the nucleus, or its containment in $\{0\}$, often yields powerful consequences for the group algebra $kG$ (or, to put it another way, the nucleus is a measure of the failure of these properties). For example, in \cite{BCRickardThick} it is shown that, when the nucleus is contained in $\{0\}$, thick subcategories of the stable module category are determined by certain subsets of $\Spec\,H^*_\text{gp}(G,k)$.

Recently in \cite{benson2023modules}, Benson and Greenlees made a number of conjectures regarding various other triangulated categories associated to the group algebra, including the DG algebra of cochains $C^*(BG,k)$, which by the Rothenberg-Steenrod construction are the derived endomorphisms of $k$ over $kG$ (see \cite{ComplexInj}). The collection of conjectures are shown to be true whenever the nucleus is contained in $\{0\}$, and many are equivalent to each other in general.

The relation of the nucleus to $C^*(BG,k)$ is a chief motivation for this paper - cochains on classifying spaces may also be considered for general topological groups, and so we consider the ``increase in dimension'' one gets by generalising finite groups to compact Lie groups.

The nucleus we define for compact Lie groups is based on the structural definition. To do this, we make the observation that $p$-nilpotency of a finite group $G$ is equivalent to the cochains $C^*(BG,\mathbb{F}_p)$ being regular as a DG algebra or ring spectrum, in the sense of Dwyer, Greenlees, and Iyengar \cite{DualityIn}.

In some sense, the passage to compact Lie groups made here is the first step: the definition is made for characteristic zero (or positive characteristics coprime to the order of the Weyl group), thus the behaviour we measure comes only from taking Lie groups of positive rank, rather than any modular representation-theoretic difficulties which occur in generality.

This assumption means the theory is almost entirely commutative-algebraic in flavour. This comes from the formality of $C^*(BG,k)$ in these hypotheses. Among many other simplifications, this tells us that regularity of cochains is equivalent to having polynomial cohomology - a condition well understood in invariant theory. This is recalled in the appendix.

In particular, our main theorem describing the nucleus (Theorem \ref{MainThm}) as the homogeneous singular locus is proven by doing the majority of the work in the more abstract context of invariant theory. That is, after defining the nucleus in Section \ref{DefinitionSection}, we classify the singular locus of the projective variety associated to a general invariant ring in Section \ref{InvSection}: there is little relevant to the precise topology and group theory of the Lie group beyond the action of the Weyl group on the torus.

Section \ref{SupportSection} then considers an analogy of a conjecture in the first version of the preprint \cite{benson2023modules} concerning the singularity category $\Dsg(C^*(BG))$ as defined by Greenlees and Stevenson \cite{MoritaAndSing}. Namely the nucleus should be the support of this category - in this section we give a reasonable notion of support and show that, in the context of Lie groups, it does indeed line up with the nucleus (indeed, it is precisely the singular locus of the cohomology ring).

\subsection{Conventions and Terminology}
Throughout this paper any grading shall be homological. A reader who objects to this is welcome to ignore the prefix ``co-'' in the word ``codegree''. In particular, differential graded (DG) algebras are monoids in the category of \textit{chain} complexes, and by the result of Shipley \cite{HZAlg} may also be considered as ring spectra. We refer to the homology/homotopy groups variously as the \textbf{coefficients}.

The various triangulated categories discussed are based on the derived category $\calD(R)$ of a ring spectrum $R$, that is the homotopy category of $R$-modules. In particular we have the category $\calD^c(R)$ of \textbf{small} or \textbf{compact} modules, which for a discrete ring $R$ are precisely those chain complexes equivalent to a bounded complex of finitely generated projectives.

A full triangulated subcategory $\mathcal{S}$ of a triangulated category $\mathcal{T}$ is \textbf{thick} if it is closed under retracts, and we denote by $\mathrm{Thick}(M)$ the smallest thick subcategory of containing $M$. We say $M$ \textbf{finitely builds} $N$ if $N\in\mathrm{Thick}(M)$, and recall that for a ring spectrum $R$, $\calD^c(R)=\mathrm{Thick}_R(R)$.

The other important subcategory at play is that of finitely generated modules, $\calD^b(R)$, called the \textbf{bounded derived category}. An exact description of this is given via normalisations in \cite{MoritaAndSing} but for most of this paper, this subcategory consists of \textit{coefficient-}finitely generated modules, that is those $M$ for which the coefficients $M_*$ are finitely generated as a graded module over $R_*$. The singularity category $\Dsg(R)$ is then the Verdier quotient of $\calD^b(R)$ by $\calD^c(R)$ - notice its thick subcategories are in correspondence with those of $\calD^b(R)$ containing $R$.

Unless otherwise stated (such as in the appendices), the tensor product and internal hom in categories are assumed to be derived. All subgroups of compact Lie groups are assumed to be closed.

\section{Group Varieties and the Nucleus}\label{DefinitionSection}
\subsection{Notation and Torus Subgroups}
We start by introducing the language of group varieties to topological groups.
\begin{defin}
	Let $G$ be a topological group and $k$ a field. We define the \textbf{group variety} $$\calV_G:=\mathrm{Spec}^h\,H^*(BG,k)$$ to be the set of homogeneous prime ideals of the cohomology ring of the classifying space $BG$.

	In general, this cohomology ring is not strictly commutative, however we can make sense of the spectrum by noting that the quotient of $H^*BG$ by its nilradical \textit{is} strictly commutative. In the context we consider ($G$ being compact Lie with Weyl group of invertible order), the cohomology ring will always be evenly graded and thus strictly commutative (and indeed a domain), and so these issues need not be considered.

	We endow $\calV_G$ with the Zariski topology. Write $I_G$ for the \textbf{irrelevant ideal}, that is the unique maximal homogeneous ideal of $H^*(BG,k)$ consisting of elements of strictly positive codegree. Considered as a geometric point in $\calV_G$, we denote $I_G$ by $0$, which is to be understood as the origin. Removing this, we obtain the projective variety $$\calP_G:=\mathrm{Proj}(H^*(BG,k))=\calV_G\setminus\{0\}$$
	By \textbf{points} of $\calP_G$ we refer to maximal ideals (that is, those maximal in $\calP_G$ as a poset with containment), and the \textbf{points} of $\calV_G$ are points of $\calP_G$ as well as the origin $0$.

	Given a subgroup $H\leq G$, we refer to the image of $\calV_H\to\calV_G$ as the \textbf{support} of $H$.
\end{defin}
\begin{exam}
	If $G\cong\mathbb{T}^n$ is a torus, then $\calV_G$ is homogeneous affine $n$-space $\mathbb{A}^n_k$. Then, removing the origin, $\calP_G$ is projective space $\mathbb{P}^{n-1}_k$, which is empty for $n=0$. If $k$ is algebraically closed, the points of $\calP_G$ are those of the form $[z_1:\dots:z_n]$ in the classical sense.
\end{exam}

\begin{lem}\label{TorusAndLinSubspace}
	Let $k$ be an algebraically closed field, and $G=\T^n$ a torus. The assignment $H\mapsto\calP_H(k)\subseteq\calP_G$ defines a surjection from the set of circle subgroups of $G$ to points of $\calP_G$ that can be expressed as $[z_1:\dots:z_n]$ with the $z_i$ lying in the image of $\mathbb{Z}$ in $k$.

	In particular, if $k$ has characteristic zero, this defines a bijection.
\end{lem}
\begin{proof}
	A map $\T\to G$ is defined by an $n$-tuple of integers $(r_1,\dots,r_n)$. For such a map to be injective, it is required that not all $r_i$ are zero, and the nonzero $r_i$ are mutually coprime. The corresponding map $\calV_\T\to\calV_G$ restricts to $\calP_\T=\text{pt}\to\calP_G$ with image $[r_1:\dots:r_n]$. We evidently get a required surjection.

	Two such tuples $(r_i)$ and $(s_i)$ of the above form define the same circle subgroup if and only if $(r_i)=\pm(s_i)$, and one can see in characteristic $0$ that a ratio $[x_1:\dots:x_n]$ of rational numbers can be expressed in this form, also uniquely up to sign, giving a bijection.
\end{proof}

We can generalise this to all torus subgroups - we instead phrase the correspondence algebraically as linear subspaces of $H_2(BT)$, knowing that these will then define hyperplanes in $\calV_G$ and $\calP_G$.

\begin{lem}\label{TheCorresp}
	Let $k$ be an algebraically closed field, and $G=\T^n$ a torus. Give $H_2(BG)\cong k^{\oplus n}$ the basis given by each factor of $G$. For $r\geq 0$, the map $T\mapsto H_2(BT)$ is a surjection from rank $r$ torus subgroups to $r$-dimensional linear subspaces of $H_2(BG)$ spanned by vectors with integral coordinates.
\end{lem}
\begin{proof}
	When $r=1$, this is as above, now defining a line rather than a point in projective space. A higher rank torus is then given by embeddings of each circle factor, that define linearly independent lines in $H_2(BG)$.
\end{proof}

\subsection{The Nucleus}
Having established the relevant language of varieties, we now consider the main definition. Throughout $k$ shall be an algebraically closed field of characteristic $p\geq0$, and $G$ a compact Lie group whose Weyl group has order not divisible by $p$.

\begin{defin}
	Let $H$ be a torus subgroup of $G$. Call $H$ \textbf{nuclear} if the cochains $C^*(BC_G(H),k)$ are not regular (as a ring spectrum in the sense of \cite{DualityIn}), where $C_G(H)$ denotes the centraliser of $H$ in $G$.

	Suppose $G$ is a compact Lie group. We define the \textbf{nucleus} of $G$ as the union of homogeneous subvarieties $$\calY_G=\bigcup_H\mathrm{Im}(\calV_H\to\calV_G)\,\subseteq\calV_G$$
	taken over nuclear torus subgroups $H$.
\end{defin}
Nuclear subgroups are currently fairly abstractly defined, however they do have a more concrete description:
\begin{lem}
	Let $H$ be a torus subgroup of $G$. The following are equivalent:
	\begin{enumerate}[i)]\item $H$ is \textbf{not} nuclear.
		\item $H^*(BC_G(H),k)$ is a polynomial ring.
		\item Given a maximal torus $T$ containing $H$, the subgroup $C_G(H)/T$ of the Weyl group $N_G(T)/T$ acts on $H_2(BT,k)$ as a pseudoreflection group.
	\end{enumerate}
\end{lem}

\begin{proof}
	The equivalence of i) and ii) is due to the formality of $C^*(BG)$ in this context - see \ref{BGFormal} and \ref{LocalSmallCor}

	The equivalence of ii) and iii) is a consequence of the Chevalley-Shepherd-Todd theorem (\cite{ShepherdTodd}, or see \cite{PolyInv} for arbitrary fields), since $$H^*(BC_G(H),k)=H^*(BT,k)^{C_G(H)/T}$$
\end{proof}

We use this to prove another useful fact about nuclear subgroups:

\begin{lem}\label{SubgroupNuclear} Let $K$ and $H$ be torus subgroups of $G$, contained in a maximal torus $T$. If $C_G(K)\subseteq C_G(H)$ and $K$ is nuclear, then $H$ is nuclear. In particular a torus subgroup of a nuclear subgroup is nuclear.
\end{lem}
\begin{proof} There is an inclusion $C_G(K)/T\subseteq C_G(H)/T$ and thus a finite extension of algebras $$A=H^*(BT)^{C_G(H)/T}\subseteq A'=H^*(BT)^{C_G(K)/T}$$
	If $H$ is not nuclear, $A$ is a polynomial ring, and thus a Noether Normalisation of $A'$. Since $A'$ is Cohen-Macaulay, $A'$ is thus a free $A$-module. It follows that $A'$ must also be polynomial so $K$ cannot be nuclear. The relevant results of invariant theory can be found in \cite{PolyInv}, 4.3.6 and 6.2.3.
\end{proof}
\begin{rem}\label{OriginSubtlety}
	In particular, this lemma implies that if any positive rank torus is nuclear, the trivial group is nuclear. Since the support of the trivial group is $\{0\}$, this means that if $0\not\in\calY_G$, then $\calY_G=\emptyset$.

	This is a subtlety in including $0$ in the definition of the nucleus. One should understand the nucleus as \textbf{trivial} if $\calY_G\subseteq\{0\}$, which is to say the only possible nuclear subgroup is the trivial subgroup. Then $\calY_G=\emptyset$ if and only if $H^*BG$ is a polynomial ring, but one can construct many examples where the nucleus is trivial but $H^*BG$ is not polynomial (see Example \ref{SegreExample}).
\end{rem}
We now state the main theorem:
\begin{thm}\label{MainThm}
	Let $k$ be an algebraically closed field, and $G$ a compact Lie group whose Weyl group has order invertible in $k$. The points of $\calY_G$ are precisely the singular points of $\calV_G$.

	Equivalently, the points of $\calY_G\setminus\{0\}$ are the singular points of $\calP_G$.
\end{thm}
\begin{rem}
	The two forms of the theorem are equivalent essentially by the observation of Remark \ref{OriginSubtlety}. $0$ is a nonsingular point of $\calV_G$ if and only if $H^*BG$ is a polynomial ring, in which case $\calP_G$ cannot have any singularities. Thus if $\calP_G$ has no singular points, the singular locus of $\calV_G$ is either $\emptyset$ or $\{0\}$ if, respectively, $H^*BG$ is polynomial or not, and the nucleus is $\emptyset$ or $\{0\}$ under the same condition.

	Hence, the proof of the theorem can safely ignore the origin point and considers only positive rank tori and the projective variety.
\end{rem}
The proof of this theorem is given at the end of Section \ref{InvSection}, and we instead consider an immediate consequence which should be compared to the c. 12 equivalent conditions of \cite{Benson1995} 1.4 for the nucleus of a finite group, when the nucleus is trivial in the sense of Remark \ref{OriginSubtlety}.
\begin{cor}
	$\calP_G$ is nonsingular if and only if for all circle subgroups $S$, the ring $H^*(BC_G(S),k)$ is polynomial.
\end{cor}
\begin{exam}\label{SegreExample}
	Consider the Lie group $G=\T^2\rtimes C_2$ with $C_2$ acting via inversion. No nontrivial torus $H$ is fixed by $C_2$ pointwise, and thus $C_G(H)=\T^2$. Thus there are no nontrivial nuclear subgroups, however the trivial subgroup is nuclear.

	Thus $\calY_G=\{0\}$. Accordingly, $\calP_G$ is nonsingular: $H^*(BG,k)=k[x,y,z]/(xy-z^2)$ with generators of codegree $2$ whose projective variety is the image under Segre embedding $\mathbb{P}^1\times\mathbb{P}^1\hookrightarrow\mathbb{P}^3$.
\end{exam}
\begin{exam}
	Now take $G=\T^3\rtimes C_2$ with $C_2$ inverting the first two factors and fixing the third. Now $H^*(BG,k)=k[x,y,z,w]/(xy-z^2)$ with generators of codegree $2$.

	There is precisely one nuclear circle subgroup - the third factor in $\T^3$. Its corresponding point in $\calP_G$ is given with coordinates $[0:0:0:1]$. The singularity here is precisely the ordinary double point of $\Spec\,k[x,y,z]/(xy-z^2)$ at the origin.
\end{exam}

As may already be evident from the above, it seems prudent to choose a maximal torus $T\cong\T^n$ and consider $N_G(T)$. We know, from Borel (e.g. \cite{BorelFibre}) and Feshbach \cite{Feshbach} in positive characteristic, that $N_G(T)\to G$ induces an isomorphism on $H^*(-,k)$, giving the same group variety. We want to replace $G$ with $N_G(T)$ (that is, assume the identity component of $G$ is a torus), and the following result states we can do so without loss of generality:

\begin{lem}\label{NormaliserSame}
	Fix a maximal torus $T$ of $G$. The nucleus of $N_G(T)$ is identical to that of $G$.
\end{lem}
\begin{proof}
	Let $H$ be a torus subgroup of $N_G(T)$. Since $T$ is Abelian, $T$ is contained in both $C_{N_G(T)}(H)$ and $C_G(H)$. Indeed, $C_{N_G(T)}(H)$ is precisely the normaliser of $T$ in $C_G(H)$, and so $H^*(BC_{N_G(T)}(H),k)\cong H^*(BC_G(H),k)$ since the Weyl group has invertible order. Thus, $H$ is nuclear when considered as a subgroup of $N_G(T)$ if and only if it is nuclear as a subgroup of $G$, and so $\calY_{N_G(T)}\subseteq\calY_G$.

	If $H'$ is a torus subgroup of $G$, then $H'$ is contained in a maximal torus $T'$, which is conjugate to $T$. The inner automorphism taking $T'$ to $T$ induces an isomorphism of normalisers, and sends $H'$ to some torus subgroup $H$ of $T$. Since inner automorphisms are homotopic to the identity, the images of the inclusions of $\calV_H$ and $\calV_{H'}$ are precisely equal. Hence the containment is an equality.
\end{proof}

Hence we can make the standing assumption: $G$ shall have identity component $T=\T^n$ and finite component group $W$, which acts on $T$ via conjugation. In this situation, $W$ acts on $H^*BT$ and hence also on $\calV_T$, and the quotient of this variety under the action is $\calV_G$. For $\p\in\calV_T$, write $W_\p$ for the isotropy subgroup of $\p$ under this action.

We shall use the following lemma in the proof of Theorem \ref{MainThm}. This result is almost trivially true over the rationals, however to generalise to other fields takes a slight amount of awkward work. We thus make heavy use of Lemma \ref{TheCorresp}.

\begin{lem}\label{TorusCover}
	Every point $\p\in \calP_G$ lies in the image of $\calP_H(k)$ for some torus $H$, such that $C_G(H)/T$ is the isotropy subgroup $W_\p$.
\end{lem}
\begin{proof}
	A point $\p\in\calP_G$ is equivalently a line in $H_2(BT,k)$, say spanned by vector $v$, and $W_\p$ is equivalently the subgroup of $W$ that fixes $v$. Let $V=H_2(BT,k)^{W_\p}$, which is a $W_\p$-subrepresentation that contains $v$. Since $W$ acts by integral matrices, $V$ is spanned by vectors with components in the integral subfield - however to avoid complications in positive characteristic we need slightly more:

	We claim that $H_2(BT,k)^{W_\p}=H_2(BT,\mathbb{Z})^{W_\p}\otimes_\mathbb{Z}k$. The former clearly contains the latter, and the former is spanned by vectors with integral components.

	If $x\in H_2(BT,\mathbb{Z})$ has image lying in $H_2(BT,k)^{W_\p}$, then for each $w\in W_\p$, $w\cdot x=x+py$ in $H_2(BT,\mathbb{Z})$ for some $y$, where $p$ is the characteristic of $k$. If $p=0$ there is nothing to prove so assume $p>0$, and let $N\in\mathbb{Z}$ be an inverse to $\left|W\right|$ modulo $p$. Then $$\bar{x}:=N\sum_{w\in W_\p}w\cdot x$$ lies in $H_2(BT,\mathbb{Z})^{W_\p}$, and has the same image as $x$ in $H_2(BT,k)$. This shows the required equality.

	Hence we can choose circle subgroups $C_1,\dots,C_r$ with classes $x_i\in H_2(BT,\mathbb{Z})^{W_\p}$ forming a basis of $H_2(BT,k)^{W_\p}$. Then $H=C_1\times\dots\times C_r$ is a torus subgroup with $H_2(BH,k)=H_2(BT,k)^{W_\p}$, which contains $v$. It thus remains to show $C_G(H)/T=W_\p$.

	By construction, any $w\in W_\p$ fixes each $C_i$ pointwise under conjugation (since their corresponding classes $H_2(BC_i,\mathbb{Z})$ are fixed) so $W_\p\subseteq C_G(H)/T$. But the reverse is true, if $w$ fixes $H$ pointwise it fixes $H_2(BH,k)$ and thus $v$.
\end{proof}

%
%
%
%

\section{Singularities in Invariant Rings}\label{InvSection}
\newcommand{\q}{\mathfrak{q}}
This section considers a more general commutative-algebraic setting of invariant rings, determining which of their prime ideals are regular. We use this to prove Theorem \ref{MainThm} at the end of the section. The main technical result is as follows:
\begin{prop}\label{MumfordButBetter}
	Let $A$ be a Noetherian commutative ring acted on by a finite group $G$, and $B=A^G$ the invariants of this ring under the action. Then for $\q\in\mathrm{Spec}\,B$, there is an isomorphism $$B^\wedge_\mathfrak{q}\cong (A_\p^\wedge)^{G_\p}$$
	where $\p\in\Spec\,A$ with $\p\cap B=\q$ and $G_\p$ denotes the subgroup of $G$ that maps $\p$ to itself.
\end{prop}
\begin{rem}
	This is by no means an original result - it can be found for example as the first theorem in Ch. X in \cite{Raynaud}.
\end{rem}
\begin{proof}
	Consider first the localisation $A_\mathfrak{q}$, of $A$ as a (finitely generated) $B$-module, at the prime $\mathfrak{q}$. This is a ring whose primes are in bijection with those of $A$ that are disjoint from $B\setminus\mathfrak{q}$. Thus, one can verify it is a semi-local ring whose maximal ideals are localisations of those in the set $$P=\{\p\in\Spec\,A:\p\cap B=\q\}$$
	of those primes over $\q$. Thus, once we complete $A_\mathfrak{q}$ as a ring with respect to its Jacobson radical $J$, the Chinese remainder theorem (see \cite{Matsumura} 8.15) gives us an isomorphism
	$$(A_\q)^\wedge_J\cong\prod_{\p\in P}A_\p^\wedge$$
	We now claim the $J$-adic topology on $A_\q$ is the same as the $\q$-adic, so that we have $(A_\q)^\wedge_J\cong A_\q^\wedge$. This follows since the Jacobson radical $J$ is the radical $\sqrt{I}$ of the ideal $I=A\cdot\q$, and $I^nA=\q^nA$.

	Thus we have an isomorphism $A_\q^\wedge\cong\prod_{\p\in P}A_\p^\wedge$. Choose some $\p\in P$, and notice $P$ is its orbit under action of $G$. Choices of coset representatives $g\in G/G_\p$ define isomorphisms $A_\p^\wedge\to A_{g.\p}^\wedge$ and thus an isomorphism $$(A_\q^\wedge\cong)\prod_{\p\in P}A_\p^\wedge\cong \prod_{g\in G/G_\p}A_\p^\wedge$$
	Under this choice of isomorphism, the coset represetatives of $G_\p$ permute the factors of the right hand side, and elements of $G_\p$ act diagonally on the product.

	Thus the diagonal embedding $(A_\p^\wedge)^{G_\p}\to\prod_{g\in G/G_\p}A_\p^\wedge$ defines an isomorphism $(A_\p^\wedge)^{G_\p}\cong (A_\q^\wedge)^G$. The result then follows as completion is exact on finitely generated modules and so $(A_\q^\wedge)^G\cong(A^G)_\q^\wedge=B_\q^\wedge$.
\end{proof}

The object now is to consider polynomial invariant rings. Let $k$ be an algebraically closed field, $G$ a finite group whose order is invertible in $k$. Let $V$ be a finite $kG$-module and $S=k[V]$ be the free graded commutative algebra generated by $\Hom_k(\Sigma^2V,k)$ - that is, a graded polynomial ring with generators in degree $-2$.

\begin{thm}\label{SingInv} Let $\p$ be a point of $\mathrm{Proj}\,S$ (corresponding to a line in $V$) and $\q=\p\cap S$. The homogeneous localisation $(S^G)_{(\q)}$ is regular if and only if the subgroup $G_\p\subseteq G$ that fixes $\p$ acts as a reflection group on $V$.
\end{thm}
\begin{proof}
	$(S^G)_{(\q)}$ is regular if and only if the non-homogeneous completion $(S^G)_\q^\wedge$ is. By Proposition \ref{MumfordButBetter} $$(S^G)_\q^\wedge\cong (S_\p^\wedge)^{G_\p}$$
	Choose a basis $e_1,\dots,e_n$ of $V$ so that $\p$ corresponds to the line spanned by $e_1$, so that $S\cong k[x_1,\dots,x_n]$ and $\p=(x_2,\dots,x_n)$.

	As a $G_\p$-representation, $V$ splits as $\langle e_1\rangle\oplus V'$, acting trivially on the first factor. Then $S_\p^\wedge\cong K[[V']]$ for some field extension $K$ of $k$, with the action of $G_\p$ induced by that on $V'$. $G_\p$ acts as a reflection group on $V$ if and only if it does on $V'$, and $K[[V']]^{G_\p}$ is regular if and only if $K[V']^{G_\p}$ is polynomial, so the result follows from the Chevalley-Shepherd-Todd theorem.
\end{proof}

We use this to prove Theorem \ref{MainThm}. For this result, our finite group is now $W$, and $V=H_2(BT)$.

\begin{proof}[Proof of Theorem \ref{MainThm}] For any $\p\in\calP_G$, Lemma \ref{TorusCover} tells us there exists a torus subgroup $H$ that $\p$ lies in the support of, and $C_G(H)/T=W_\p$. If the local ring at $\p$ is singular, then $W_\p$ does not act as a pseudoreflection group by Theorem \ref{SingInv}, which means $H$ is a nuclear subgroup. Thus all singularities lie in the nucleus.

	Conversely, if $\p$ lies in the support of a nuclear subgroup $K$, we can again choose $H$ such that $C_G(H)/T=W_\p$. Since $C_G(K)/T$ fixes $\p$, $C_G(K)/T$ is a subgroup of $W_\p$. That $H$ is itself nuclear then follows from Lemma \ref{SubgroupNuclear}, and then $p$ is a singular point by Theorem \ref{SingInv}.
\end{proof}
\section{Support of Singularity Categories}\label{SupportSection} In this section we make a definition of the support of a singularity category. As will be discussed in Remark \ref{42Rem} this notion of singular support has been used before to classify thick subcategories for hypersurface rings, however to the author's knowledge has not been applied to ring spectra.

The motivation for this comes from the first version of the preprint \cite{benson2023modules}, in which Conjecture 7.5 states that the nucleus of a finite group $G$ is the same as the support of the singularity category $\Dsg(C^*(BG))$. This category indeed admits a central action by the cohomology ring $H^*(BG)$, but lacks arbitrary coproducts, so one cannot localise with respect to a prime within the category itself. Thus one cannot directly apply the notion of support of Benson, Iyengar, and Krause \cite{BIKSupport}, and must consider a method regarding $\Dsg(R)$ as a subquotient of $D(R)$. It should be noted that this issue is circumvented in work of Stevenson \cite{SingSubCatsGreg} with discrete rings by considering a larger category $S(R)=K_\text{ac}(\mathrm{Inj}\,R)$ with all sums, whose subcategory of compact objects is $\Dsg(R)$, up to summands.
\subsection{A Singular Theory of Support}
Let $R$ be a commutative ring spectrum augmented over a field $k$, such that $R_*$ is a Noetherian ring. Write $R\subP$ for the homotopical localisation at a homogeneous prime $\p$ of $R_*$, recalling $(R\subP)_*=(R_*)\subP$, the homogeneous localisation of the graded ring $R_*$.

\begin{rem}\label{fgExplanation}
	The requirement of augmentation over $k$ is only to give a well-defined notion of a singularity category, as in \cite{MoritaAndSing}. For formal DGAs, ``finitely-generated'' modules should probably just be in the coefficient case. Indeed we remove the augmentation requirement and when $R$ is formal define the singularity category as the Verdier quotient of the category of coefficient-finitely generated modules by those small ones. In the case of nucleus, coconnected augmented algebras are coefficient-normalisable by Noether's normalisation lemma, and so is consistent with how it is defined in \cite{MoritaAndSing}.
\end{rem}

\begin{defin}
	Let $M$ be an $R$-module, finitely generated in the sense explained in Remark \ref{fgExplanation}. Define the \textbf{singular support} of $M$, $\mathrm{Supp}_\text{sg}(M)$, to be the set of homogeneous primes $\p$ of $R_*$ for which $M\otimes_RR\subP$ is not a small object of $\calD(R\subP)$.

	If two modules are isomorphic in the singularity category, then their singular supports are equal, since $-\otimes_R R\subP$ sends small objects of $\calD(R)$ to small objects of $\calD(R\subP)$.
\end{defin}

\begin{defin}
	The \textbf{support} of the singularity category $\Dsg(R)$ is the union $$\mathrm{Supp}\,\Dsg(R)=\bigcup_{M\in\Dsg(R)}\mathrm{Supp}_\text{sg}(M)\subseteq\Spec^h(R_*)$$
\end{defin}

We want to understand this when $R$ is a formal DGA over $k$. To start first, we consider the case that $R$ is a discrete ring.

\begin{prop}\label{discreteSupport}
	Let $R$ be a discrete Noetherian ring. The support of the singularity category of $R$ is the singular locus $\mathrm{Sing}(R)$, the set of primes $\p$ for which $R_\p$ is not regular.
\end{prop}
\begin{proof}
	If $R_\p$ is regular, then $\Dsg(R_\p)=0$ and so $\p$ cannot lie in any singular support.
	If $R_\p$ is not regular, then the residue field $k_\p$ is finitely generated but not small. Then $\p\in\mathrm{Supp}_\text{sg}(R/\p)$ since $(R/\p)\otimes_R R_\p=k_\p$.
\end{proof}

We shall now extend this to when $R$ is a formal DGA. In fact, it turns out the proof is identical: we just need to find a ``witness'' to the singularity of $R\subP$. That is, if $\Dsg(R\subP)\neq0$ then some nontrivial element lies in the image of $-\otimes_RR\subP$.

We first note the following, without proof. This has been used implicitly in \cite{ThickSubcats}:

\begin{prop} If $R$ is a formal DGA then for any $\p\in\Spec^h(R)$, the DGA $R\subP$ is also formal.
\end{prop}
Since $R\subP$ is formal, we can therefore use Corollary \ref{LocalSmallCor}, tying the vanishing of the singularity category of formal DGAs to regularity in the commutative-algebraic sense.
\begin{thm}
	Let $R$ be a Noetherian commutative graded ring, considered as a DG algebra with zero differential. The singular support of $R$ is the homogeneous singular locus $\mathrm{Sing}^h(R)=\mathrm{Sing}(R)\cap\Spec^h(R)$.
\end{thm}
\begin{proof}
	If $R\subP$ is regular then for $M$ a (coefficient-) finitely generated DG module, $M\otimes_R R\subP$ is finitely generated over $R\subP$, also formal, and thus by Corollary \ref{LocalSmallCor} is small, thus $\p$ is not in the singular support of any module.

	If $R\subP$ is not regular, then consider $R/\p$ as a DG $R$-module with zero differentials. $R/\p\otimes_RR\subP$ is then the graded residue field $\kappa(\p)$. If this is a small $R\subP$-module, then $R\subP$ would be regular. Thus $\p\in\mathrm{Supp}_\text{sg}(R/\p)$.
\end{proof}
\begin{cor}
	The nucleus of a compact Lie group, over a field in which the order of the Weyl group is invertible, is precisely the support of the singularity category of its cochains.
\end{cor}
\begin{rem}\label{42Rem}
	Having obtained a reasonable notion of support for singularity categories, it is reasonable to ask what this could measure. This notion of support has been used before in work of Takahashi \cite{SingSubCatsRyo} and Stevenson \cite{SingSubCatsGreg} in order to classify thick subcategories of $\Dsg(R)$ for a (discrete) hypersurface ring $R$. Indeed, they show that such subcategories are in bijection with non-empty specialisation-closed subsets of $\mathrm{Sing}(R)$.
	
	This does not extend beyond hypersurfaces: if $(R,\m)$ is an Artinian local complete intersection of codimension $c>1$, Carlson and Iyengar \cite{ThickSubcats} show such subcategories are in bijection with non-empty specialisation-closed subsets of $\Spec^h(S)$, where $S$ is a graded polynomial ring on $c$ generators. Thus, there are far more thick subcategories of $\Dsg(R)$ than there are subsets of $\{\m\}$.
\end{rem}

\subsection{Conjectures on Generation}
We now consider two conjectures of \cite{benson2023modules}, transported to Lie groups. Recall we assume $G$ is a compact Lie group, $T$ a \textit{normal} maximal torus, $W=G/T$ the Weyl group, and $k$ is a field over which $\left|W\right|$ is invertible.

We consider the map $i:C_*T\to C_*G$ of DG algebras, and the corresponding restriction/(co)induction adjunctions $i_*\dashv i*\dashv i_!$.

Note also, as a $C_*T$-module, $C_*G$ is a free module $$i^*C_*G\cong\bigoplus_{w\in W}C_*T$$ and thus the (co)induction functors may be considered as non-derived. In particular, $i_*(k)=C_*W$, the standard group algebra. Due to $T$ having finite index in $G$, (co)induction in this situation enjoys the same properties as for finite groups:
\begin{lem}
	Induction and coinduction coincide. That is, there is a natural equivalence $i_*\simeq i_!$.
\end{lem}
\begin{proof} Note first that projection onto the identity factor gives us a $C_*T$-module morphism $\rho:i^*C_*G\to C_*T$, of which $i$ is a section. The adjoint map $\hat\rho:C_*G\to\Hom_{C_*T}(C_*G,C_*T)$ is then an isomorphism. We then consider, for $C_*T$-modules $M$, the composite
	$$C_*G\otimes_{C_*T}M\to \Hom_{C_*T}(C_*G,C_*T)\otimes_{C_*T}M\xrightarrow{\text{composition}}\Hom_{C_*T}(C_*G,M)$$
	which is a natural transformation $\varphi_M:i_*(M)\to i_!(M)$.

	Both functors are exact, and commute with all coproducts ($i_!$ since $C_*G$ is a small $C_*T$-module). Thus the collection of $C_*T$-modules $M$ for which $\varphi_M$ is an equivalence is localising. But since it also evidently contains $C_*T$ itself, it is the entire category of $C_*T$-modules.
\end{proof}

We now state the two conjectures. Recall by the Rothenberg-Steenrod construction $C^*BG\simeq\Hom_{C_*G}(k,k)$. In the original conjectures, $G$ is a finite group and $T$ is a Sylow-$p$-subgroup, where the chains $C_*G$ are just the ordinary group algebra.

\begin{conj}[cf. \cite{benson2023modules} 1.6]\label{conj16}
	$\calD^b(C^*BG)$ is generated as a thick subcategory of $\calD(C^*BG)$ by the objects $\Hom_{C_*G}(k,X)$ for $X$ in $\calD^b(C_*G)$.
\end{conj}
\begin{conj}[cf. \cite{benson2023modules} 1.7]\label{conj17}
	The bounded derived category $\calD^b(C^*BG)$ is generated by the module $C^*BT$, where $T$ is a maximal torus of $G$.
\end{conj}
\begin{rem} Conjecture \ref{conj17} can be put in a more general context of invariant rings in the same way as Section \ref{InvSection}: the statement is the bounded derived category of $k[V]^G$ is generated by $k[V]$.
\end{rem}

The proof of the equivalence of these two conjectures is, perhaps unsurprisingly, almost identical to the corresponding ones for finite groups:

\begin{lem}
	Conjectures \ref{conj16} and \ref{conj17} are equivalent.
\end{lem}
\begin{proof}
	Note $C_*W\simeq i_*(k)\simeq i_!(k)$, and so $\Hom_{C_*G}(k,C_*W)\simeq\Hom_{C_*T}(k,k)=C^*BT$. Since $C_*W$ is a bounded $C_*G$ module, this proves that \ref{conj16} implies \ref{conj17}.

	Conversely, we need to show any object of the form $\Hom_{C_*G}(k,X)$ with $X\in\calD^b(C_*G)$ is built from $C^*BT$. It thus suffices to show any $X\in\calD^b(C_*G)$ is finitely built from $C_*W$. Indeed, $X$ is finitely built from $k$ (since $H_*G$ is finite dimensional over $k$ and $H_0G=kW\to k$ has nilpotent kernel, see \cite{DualityIn} 3.16). But as the order of $W$ is invertible, $C_*W$ finitely builds $k$ over itself, thus also over $C_*G$.
\end{proof}

We end this section with a condition under which the conjecture holds.
\begin{prop}
	Suppose $\calD^b(C^*BG)$ is generated by formal modules (that is, under the equivalence $C^*BG\simeq H^*BG$, those modules $M$ that are equivalent to modules with zero differential). Then Conjecture \ref{conj17} holds.
\end{prop}
\begin{proof} We explicitly work over $H^*BG$ with zero differential, noting $H^*BT$ is then a finite formal module. Given a formal module $M$, consider the \textit{non-derived} tensor product $M\otimes_{H^*BG}H^*BT$ as graded modules. Considered as a formal $H^*BT$-module, this is finitely generated and thus as $H^*BT$ is regular, it is finitely built.

	So $H^*BT$ finitely builds $M\otimes_{H^*BG}H^*BT$ over $H*BG$. But over $H^*BG$, $H^*BG$ is a retract of $H^*BT$, and so $M$ is a retract of $M\otimes_{H^*BG}H^*BT$.

	So all formal modules are generated by $H^*BT$.
\end{proof}
\appendix
\section{Formality}
Central to our discussion is the fact that, in our context, the DG algebra $C^*(BG)$ is formal. That is, it is quasi-isomorphic as a DG algebra to a DG algebra with zero differentials. Indeed, we state that here:
\begin{thm}\label{BGFormal}
	Let $G$ be a compact Lie group and $k$ a field over which the Weyl group of $G$ has invertible order. Then $C^*(BG,k)$ is a formal DG algebra.
\end{thm}
\begin{proof}[Sketch Proof] When $k$ is of positive characteristic, this is proven in \cite{BGFormal}. For characteristic zero, this is well-known, especially when $G$ is connected (thus $H^*(BG)$ is polynomial). When $G$ is not connected, one can observe that $C^*(BG_e)$ is formal where $G_e$ is the identity component, and that $C^*(BG)$ is the homotopy fixed points of $C^*(BG_e)$ under the action of the component group $G/G_e$. In characteristic zero, homotopy fixed points and ordinary fixed points coincide, and details notwithstanding this proves formality.
\end{proof}

For the purposes of this paper, the commutative-algebraic behaviour we have for formal DG algebras can be summarised as the following proposition:
\begin{prop}\label{LocalSmall}
	Let $(R,\m)$ be a Noetherian local graded ring, so that $\kappa=R/\m$ is a graded field. Let $M$ be a finitely generated graded module of $R$. The following are equivalent:
	\begin{enumerate}[i)]
		\item $M$ is small as a differential graded $R$-module.
		\item $M$ has a finite free resolution as a graded $R$-module.
		\item The ungraded localisation $M_\m$ has a finite free resolution as an $R_\m$-module.
		\item The ungraded completion $M_\m^\wedge$ has a finite free resolution as an $R_\m^\wedge$-module.
	\end{enumerate}
\end{prop}
\begin{rem}
	The latter three conditions are fairly standard and are only included for convenience earlier on. See for example \cite{BrunsHerzog} \S1.5. The equivalence of i) and ii) is the non-trivial result.
\end{rem}
\newcommand{\DG}{\mathbf{DG}}
\newcommand{\Ch}{\mathbf{Ch}}
To sketch the proof of this equivalence we must first discuss the \textbf{totalisation functor} from the category $\Ch_R$ of chain complexes of graded $R$-modules and $\DG_R$ of differential graded $R$-modules. That is, the functor $T:\Ch_R\to\DG_R$ which sends a complex $C$ to the DG module $T(C)$ with $$T(C)_n=\bigoplus_{i+j=n}C_{ij}$$ with differentials taken termwise from $C$, with a possible change of sign. One can observe the homology of $T(C)$ is the total complex of the bigraded homology of $C$, and thus quasi-isomorphisms are preserved leading to an induced exact functor $$T:\calD_\Ch(R)\to\calD_\DG(R)$$ between the two corresponding derived categories.

The crucial observation is that this is in fact a \textit{tensor-triangulated} functor, preserving each derived tensor-product in that $$T(C\otimes_\DG D)\simeq T(C)\otimes_\Ch T(D)$$
We do not require this in full force, we only need this for graded modules concentrated at $0$.
\begin{lem}\label{DGTensorProd}
	Let $M$ and $N$ be graded $R$-modules. Considered as formal DG $R$-modules, their derived tensor product has homology $$H_n(M\otimes_\DG N)=\bigoplus_{p+q=n}\mathrm{Tor}_{pq}^R(M,N)$$
\end{lem}
\begin{proof}[Sketch Proof]
	Take a free resolution of $P\in\Ch_R$ of $M$. We obtain a corresponding map  $T(P)\to M$ of DG modules, which is now a \textit{semi-free} resolution of $M$, in particular a cofibrant replacement. The derived tensor product with $M$ is then, in the respective categories, the non-derived product with $P$ and $T(P)$, which one can verify by hand is preserved by $T$.
\end{proof}
This is the key to Proposition \ref{LocalSmall}:
\begin{proof}[Sketch Proof of Proposition \ref{LocalSmall}]
	As observed, the equivalences of ii), iii), and iv) are standard.

	The equivalence of i) and ii) comes from the observation that smallness is measured by the derived tensor product of a module with $\kappa$. That is, $M$ has a finite free resolution if and only if $\mathrm{Tor}^R_{**}(M,\kappa)$ is finite dimensional over $\kappa$, and is a small DG module if and only if $M\otimes_\DG\kappa$ has finite dimensional homology over $\kappa$, and so the equivalence comes from taking $N=\kappa$ in \ref{DGTensorProd}.
\end{proof}
\begin{rem}
	The assertion that $\kappa\otimes-$ measures smallness for DG modules is given, when $\kappa$ is an ordinary non-graded field, in Theorem 4.8 in \cite{HomOfPerfect}, and this paper also gives the equivalence of i) and ii) when in addition $R$ is graded non-negatively (or non-positively). That we can extend the result to our context is noting that the construction of minimal semifree resolutions also works when $\kappa$ is a graded field.
\end{rem}
Our principal use of \ref{LocalSmall} is the following:
\begin{cor}\label{LocalSmallCor}
	Let $(R,\m,\kappa)$ be a Noetherian local graded ring. $R$ is regular as a formal DG algebra (in that all DG modules with finitely generated homology are small) if and only if $R$ is regular as a local graded ring.

	If $\kappa=k$ is a field, then this is equivalent to regularity as a ring spectrum in the sense of \cite{DualityIn}, and if $R$ is an algebra generated over $k$ by all positive/all negative elements, this is equivalent to $R$ being a polynomial ring.
\end{cor}
\begin{proof} Take $M=\kappa$ in \ref{LocalSmall}: $\kappa$ is a small DG module if and only if it has a finite free resolution. The latter condition is equivalent to being regular as a local graded ring by Auslander-Buchsbaum-Serre, thus it suffices to show a DG module $M$ with finitely-generated homology is small if $\kappa$ is small.
	Indeed, if $\kappa$ is small, $R$ is regular as a local graded ring and so $H_*M$ has finite projective dimension $d$, over which we induct.

	If $d=0$, then $H_*M$ is a free module, thus $M$ is also, hence is small. Otherwise, choose a surjection $F\to H_*M$ where $F$ is a finitely generated free $R$-module.

	This defines a map $F\to M$ of DG modules, and let $K$ be its fibre. Evidently $H_*K$ has projective dimension $d-1$ thus $K$ is small. Since $F$ is small, it follows $M$ is too.\\

	The final two statements are clear: for the penultimate note regularity for an augmented ring spectrum $R\to k$ is precisely that $k$ is small over $R$, and the last one is a standard fact of graded rings.
\end{proof}
\begin{rem}\label{HomoAndCoeffRem}
	We have shown that for a Noetherian graded ring $R$ augmented over a field $k$, the homotopical notion of regularity of \cite{DualityIn} coincides with the ``coefficient'' notion. Another way of showing this is that, at least in some cases, the total complex functor $T$ preserves not only the derived tensor product but also the derived internal hom, possibly using the direct product total complex instead. That is, one can show there is an expression $$H_n\Big(\Hom_\DG(M,N)\Big)=\prod_{p+q=n}\mathrm{Ext}_R^{pq}(M,N)$$ analogous to \ref{DGTensorProd}. Taking $M=N=k$ shows that homotopical and coefficient notions of regularity coincide for formal $R$, and taking $M=k$ and $N=R$ show the notions coincide for Gorenstein rings also in the sense of \cite{DualityIn}. For Lie groups, this leads to an interesting comparison of the orientability of the adjoint representation to Watanabe's theorem (cf. \cite{PolyInv} 4.6.2) on when invariant rings are Gorenstein.
\end{rem}

\printbibliography
\end{document}